\documentclass[11pt]{amsart} 

\usepackage[a4paper,margin=2.5cm]{geometry}
\usepackage{enumitem}

\usepackage[normalem]{ulem}

\usepackage[colorlinks=true,urlcolor=blue,citecolor=red,linkcolor=blue,linktocpage,pdfpagelabels,bookmarksnumbered,bookmarksopen]{hyperref}

\usepackage[T1]{fontenc}
\usepackage[utf8]{inputenc}
\usepackage{lmodern}
\usepackage{mathrsfs}

\usepackage[colorinlistoftodos]{todonotes}
\makeatletter
\providecommand\@dotsep{5}
\def\listtodoname{List of Todos}
\def\listoftodos{\@starttoc{tdo}\listtodoname}
\makeatother
\allowdisplaybreaks

\newcommand{\eps}{\varepsilon}

\newcommand{\R}{\mathbb{R}}

\newcommand{\RN}{{\mathbb{R}^N}}
\newcommand{\RD}{{\mathbb{R}^2}}

\newcommand{\weakto}{\rightharpoonup}

\DeclareMathOperator{\meas}{meas}
\DeclareMathOperator{\supp}{supp}
 
\renewcommand{\le}{\leslant}
\renewcommand{\ge}{\geslant}
\renewcommand{\a }{\alpha }

\renewcommand{\d }{\delta }

\renewcommand{\l }{\lambda}

\newcommand{\n }{\nabla }
\newcommand{\s }{\sigma }

\renewcommand{\O}{\Omega}

\newcommand{\HV}{H^1_V(\RD)}

\newcommand{\Ne}{\mathcal{N}}

\newcommand{\N}{\mathbb{N}}

\newcommand{\ird }{\int_{\RD}}
\newcommand{\irn }{\int_{\RN}}

\def\bbm[#1]{\mbox{\boldmath $#1$}}
\newcommand{\beq }{\begin{equation}}
\newcommand{\eeq }{\end{equation}}

\renewcommand{\le}{\leq}
\renewcommand{\ge}{\geq}

\newtheorem{theorem}{Theorem}[section]
\newtheorem{lemma}[theorem]{Lemma}

\newtheorem{definition}[theorem]{Definition}
\newtheorem{proposition}[theorem]{Proposition}
\newtheorem{corollary}[theorem]{Corollary}

\theoremstyle{definition}
\newtheorem{remark}[theorem]{Remark}

\allowdisplaybreaks

\title[Embedding of weighted Sobolev spaces and applications]{On the embedding of weighted Sobolev spaces \\with applications to a planar\\ nonlinear Schr\"{o}dinger equation}
\author[A. Azzollini]{Antonio Azzollini}
\address{A. Azzollini \newline\indent
Dipartimento di Matematica, Informatica ed Economia, \newline\indent Universit\`a degli
	Studi della Basilicata,
	\newline\indent
	Via dell'Ateneo Lucano 10, 85100
	Potenza, Italy}
\email{antonio.azzollini@unibas.it}

\author[A. Pomponio]{Alessio Pomponio}
\address{A. Pomponio
\newline\indent Dipartimento di Meccanica, Matematica e Management,\newline \indent
	Politecnico di Bari
	\newline\indent
	Via Orabona 4,  70125  Bari, Italy}
\email{alessio.pomponio@poliba.it}

\author[S. Secchi]{Simone Secchi}
\address{S. Secchi
\newline\indent Dipartimento di Matematica e Applicazioni \newline\indent
Università degli Studi di Milano - Bicocca
	\newline\indent
Via Roberto Cozzi 55, 20125  Milano, Italy}
\email{simone.secchi@unimib.it}

\thanks{A.A. and A.P. are partially supported by  INdAM - GNAMPA Project 2024 ``Metodi variazionali e
topologici per alcune equazioni di Schrodinger nonlineari" CUP E53C23001670001.
A.P. is partially financed by European Union - Next Generation EU - PRIN 2022 PNRR ``P2022YFAJH Linear and Nonlinear PDE's: New directions and Applications". 
S.S. is partially supported by INdAM - GNAMPA Project 2024 ``Aspetti geometrici e analitici di alcuni problemi locali e non-locali in mancanza di compattezza'' CUP E53C23001670001}


\begin{document} 

\begin{abstract}
	In this paper we study the embedding's properties for the weighted Sobolev space $H^1_V(\RN)$ into the Lebesgue weighted space $L^\tau_W(\RN)$. Here $V$ and $W$ are diverging weight functions. 
	The different behaviour of $V$ with respect to $W$ at infinity plays a crucial role. Particular attention is paid to the case $V=W$. This situation is very delicate since it depends strongly on the dimension and, in particular, $N=2$ is somewhat a limit case. As an application, an existence result for a planar nonlinear Schr\"odinger equation in presence of coercive potentials is provided.
\end{abstract}

 \maketitle 

\section{Introduction}

This note is devoted to the description of some embedding theorems for weighted Sobolev spaces in  $\mathbb{R}^N$ for a class of diverging weight functions. More precisely, we consider two continuous  potentials $V\colon \RN\to \R$ and $W \colon \mathbb{R}^N \to \mathbb{R}$  bounded below by a positive constant and diverging at infinity. 
We investigate the continuous and the compact embeddings of the $H^1_V(\RN)$ into $L^\tau_W(\RN)$ where $H_V^1 (\mathbb{R}^N)$ is the weighted Sobolev space defined as the completion of $C_0^\infty(\R^N)$ with respect to the norm
\begin{displaymath}
	\Vert u \Vert_V := \left( \int_{\mathbb{R}^N} \left( \vert \nabla u \vert^2 + V(x) \vert u \vert^2 \right) \, dx 
	\right)^{1/2},
\end{displaymath}
and $L_W^\tau (\mathbb{R}^N)$ is the weighted Lebesgue space
\begin{displaymath}
L^\tau_W(\RN):=\left\{u\in \mathscr{M}(\RN) \mid \int_{\mathbb{R}^N} W(x) \vert u \vert^{\tau} \, dx<+\infty\right\},
\end{displaymath}
endowed with the norm
\[
\|u\|_{W,\tau}:=\left(\int_{\mathbb{R}^N} W(x) \vert u \vert^{\tau} \, dx\right)^{\frac 1\tau}.
\]
The symbol~$\mathscr{M}(\mathbb{R}^N)$ denotes the set of all measurable
functions from $\mathbb{R}^N$ to $\mathbb{R}$.

The embedding properties between weighted Sobolev spaces, eventually also with a potential for the gradient term in norm, and weighted Lebesgue spaces have been extensively investigated: we refer the reader to the monograph \cite{opic} and the references therein. In this note we focus on the case of two potentials, both coercive, and a potential equal to one in the gradient term of the norm, for which we were unable to find useful results in the literature. 
Our aim is, therefore, to analyse this situation and apply our results to  the Schr\"odinger equation.

\bigskip

Roughly speaking, whenever $W$ \emph{dominates} $V$ at infinity, in general we cannot expect any embedding between $H_V^1 (\mathbb{R}^N) $ and $L_W^\tau (\mathbb{R}^N)$ and so we need to restrict our attention to the case of two potentials with a similar behaviour at infinity or to that in which $V$ \emph{prevails} over $W$ at infinity.  In particular, we first consider suitable conditions on $V$ and $W$ which ensure that 
\[
	\lim_{|x|\to +\infty}\frac{V(x)}{W(x)}=+\infty.
\]
This case is easier and, in Theorem \ref{thvwa}, we can prove that $H_V^1 (\mathbb{R}^N) $ is compactly embedded into $L_W^\tau (\mathbb{R}^N)$, for suitable $\tau$.

On the other hand, the case in which
\begin{equation}\label{eq:chall}
    \liminf_{|x|\to +\infty}\frac{V(x)}{W(x)}\in\R
\end{equation}
appears to be challenging, even in the simplest case $V=W$. To the best of our knowledge, the only result in the literature is contained in \cite{ADP}. 
In \cite[Theorem 2.2]{ADP}, in order to deal with a planar  Schr\"odinger equation with competing logarithmic self-interactions, the authors prove that, under a suitable relation between $V$ and its gradient, $H_V^1 (\mathbb{R}^2)$ is continuously embedded into $L_W^\tau (\mathbb{R}^2)$, for any $\tau \ge 2$. A precise statement appears in Theorem \ref{th:emb}. 
We point out that the analysis performed in \cite{ADP}  is restricted only to the continuous embedding and dimension $N=2$. \\
One of the main aims of the present paper consists in enlarging the investigation initiated in \cite{ADP}, studying the influence of the dimension in both continuous and compact embedding of $H_V^1 (\mathbb{R}^N)$ in $L_V^\tau (\mathbb{R}^N)$.\\
Taking into account \cite[Theorem 2.2]{ADP},  by Theorem \ref{Vnon}, Theorem \ref{pr:general} and Theorem \ref{th:main1}, we conclude that, for a considerable large class of positive and diverging potentials, we have:
	\begin{itemize}
		\item if $N\ge 3$, $H_V^1 (\mathbb{R}^N)$ is \emph{not} included into $L_V^\tau (\mathbb{R}^N)$, for any $\tau > 2$;
		\item if $N=2$, $H_V^1 (\mathbb{R}^2)$ is continuously but \emph{not} compactly embedded into $L_V^\tau (\mathbb{R}^2)$, for any $\tau > 2$;
		\item if $N=1$, $H_V^1 (\mathbb{R})$ is compactly embedded into $L_V^\tau (\mathbb{R})$  for any $\tau > 2$  and in $L^\infty(\R)$.
	\end{itemize}

Roughly speaking, the presence of a coercive potential $V$ \emph{improves} the usual Sobolev embedding theorem in dimension $N=1$, leaves it \emph{unchanged} in dimension $N=2$, while it basically \emph{destroys} the embedding in dimension $N \geq 3$.

We can say that $N = 2$ is a sort of limit case for the embedding properties of $H_V^1 (\mathbb{R}^N)$ into
$L_V^\tau (\mathbb{R^N})$.




\medskip

In addition, we analyse  the  case in which $V$ and $W$ are both radially symmetric functions. Denoting by $H^1_{V,{\rm rad}}(\RN)$ and $L^\tau_{W,{\rm rad}}(\RN)$, respectively, the subsets of radial functions contained in $H^1_{V}(\RN)$ and in $L^\tau_{W}(\RN)$, by means of the Strauss Radial Lemma (see \cite{strauss}) we show in Theorem \ref{thrad} that $H^1_{V,{\rm rad}}(\RN)$ is compactly embedded into $L^\tau_{W,{\rm rad}}(\RN)$, for suitable $\tau$, and for a large class of potentials. In particular, under certain conditions, $W$ could also dominate $V$ at infinity.

\bigskip

In light of  the considerations yielded in the first part, in the second part of the paper we focus our attention on semilinear Schr\"odinger equation with external potentials. One can find many papers about the problem
\begin{equation} \label{eq:2}
	-\Delta u + V(x)u = W(x) \vert u \vert^{p-1}u \quad \hbox{in $\mathbb{R}^N$},
\end{equation}
where $V$ and $W$ are continuous real-valued functions. If we look for solutions $u \in H^1(\mathbb{R}^N)$ of \eqref{eq:2} by means of variational methods, the lack of compactness of the embedding $H^1(\mathbb{R^N}) \subset L^p(\mathbb{R}^N)$ for $p<2N/(N-2)$ is typically overcome by adding suitable assumptions on $V$ and $W$. Rabinowitz considered in \cite{rabinowitz} the case $W\equiv 1$ with the requirement 
\begin{displaymath}
	0<\inf_{x \in \mathbb{R}^N} V(x) <  \liminf_{\vert x \vert \to +\infty} V(x).
\end{displaymath}
The case of \emph{competing} potentials $V$ and $W$ is harder to deal with, and most existence results require $V$ and $W$ to behave essentially in \emph{opposite} ways. Roughly speaking, by P.L. Lions' concentration-compactness principle, sequences of  ``almost critical'' points of the Euler functional associated to \eqref{eq:2} fail to be compact because they slide off to infinity. Compactness may be restored by assuming that $V$ and $W$ converge to some finite limits at infinity in such a way that this behaviour is not convenient for minimizing sequences. See, for example, \cite{CP} and the references therein.

The borderline case $V \equiv W$ under the assumption that $V$ diverges at infinity turns out to be very hard to handle: we do not have any problem at infinity to use as a barrier, and we do not have suitable compact embeddings in order to ensure the convergence of minimizing or Palais-Smale sequences. The construction of a solution $u \in H^1(\R^N)$ to 
\begin{equation*} 
	-\Delta u + V(x)u = V(x) \vert u \vert^{p-1}u 
\end{equation*}
when $\lim_{\vert x \vert \to +\infty} V(x)=+\infty$ remains essentially open. However, the continuous embedding of  $H_V^1 (\mathbb{R}^2)$  into $L_V^{p+1} (\mathbb{R}^2)$ obtained in the first part of the paper suggests an \emph{ansatz} to overcome such difficulties.
Then we consider in \eqref{eq:2} couples of potentials differing for a constant reducing the equation to the form
	\begin{equation}\label{eq:11}
	-\Delta u +( V(x)-\l) u=V(x)|u|^{p-1}u,\quad\hbox{ in }\RD, 
	\end{equation} 
for some $\l\in \R$, whose solutions could be found as critical points of the functional
\begin{displaymath}
	E(u) = \frac{1}{2} \int_{\mathbb{R}^2} \left( \vert \nabla u \vert^2 + \left( V(x) - \lambda \right)  \vert u \vert^2 \right)\, dx - \frac{1}{p+1} \int_{\mathbb{R}^2} V(x) \vert u \vert^{p+1} \, dx.
\end{displaymath}
Even if  the lack of the compact embedding of $\HV$ into $L^{p+1}_V(\RD)$ makes the variational approach to the problem not trivial,  in Theorem \ref{th:main} we are able to provide an existence result for \eqref{eq:11} by reducing the equation to an eigenvalue problem for the nonlinear operator
	\begin{displaymath}
		u\in H_V^1 (\mathbb{R}^2)\mapsto-\Delta u + V(x)(1-|u|^{p-1})u\in (H_V^1 (\mathbb{R}^2))'.
	\end{displaymath}

\medskip

The paper is organized as follows. In Section \ref{se:func} we study the embedding properties of $H^1_V(\RN)$ into $L^\tau_V(\RN)$ under various behaviour at infinity of the potentials $V$ and $W$. In Section \ref{se:exi} we deal with the nonlinear eigenvalue problem \eqref{eq:11}.

\section{Embedding properties of some weighted Sobolev spaces}\label{se:func}

We introduce a practical shorthand to describe the class of potential functions we will be working with in this paper.

\begin{definition}
	We will say that a function $V \colon \RN \to \R$ belongs to $\mathscr{C}^*(\RN)$ if $V$ is continuous and $\inf_{x \in \RN} V(x) >0$.
\end{definition}

In this section we study the embeddings of $H^1_V(\RN)$ into
$L^\tau_W(\RN)$ assuming that $V$ is unbounded and both $V$ and $W$
 belong to $\mathscr{C}^*(\RN)$.

\subsection{The general setting}\label{sec:general}

\

As a first step we consider the case of two potentials in $\mathscr{C}^*(\RN)$ with different
behaviour at infinity. We will see that if $V$ {\em dominates} $W$ at
infinity, namely under a condition guaranteeing $V(x)/W(x) \to +\infty$
as $\vert x \vert \to +\infty$, then $H^1_V(\RN)$ is compactly
embedded into $L^\tau_W(\RN)$, for suitable $\tau$. More precisely,
the following holds.

\begin{theorem}\label{thvwa}
  Let $N \geq 1$, $V \in \mathscr{C}^*$, $W \in \mathscr{C}^*$, and suppose that there exist a number~$\a\in (0,1)$
  and two positive constants $c$ and $C$ such that
\begin{equation}\label{VWa}\tag{$\mathcal{VW}$}
0<c\le W(x)\le C \big(V(x)\big)^\a \quad \text{ for all } x\in \RN.
\end{equation}
If \( \lim_{\vert x \vert \to +\infty}V(x)=+\infty \),
then $H^1_V(\RN)$ is compactly embedded in $L^\tau_W(\RN)$ for
\begin{equation} \label{tau}
\begin{cases}
\tau \ge 2 & \text{if }N=1,2,
\\
2\le \tau < \dfrac{2N-4\a}{N-2} & \text{if }N\ge 3.
\end{cases}
\end{equation}
The embedding is only continuous if $N\ge 3$ and $\tau = \frac{2N-4\a}{N-2}$.
\end{theorem}

\begin{proof}
Let $u\in  H^1_V(\RN)$ and take $\tau$ as in \eqref{tau}.
We have 
\begin{align*}
\|u\|_{W,\tau}^\tau&=\irn W(x)|u|^\tau\, dx
=\irn\frac{W(x)}{\big(V(x)\big)^\a}\big(V(x)\big)^\a|u|^{2\a}|u|^{\tau-2\a}\, dx
\\
&\le C \left(\irn\left(\big(V(x)\big)^\a|u|^{2\a}\right)^\frac{1}{\a}\, dx\right)^\a
\left(\irn|u|^\frac{\tau-2\a}{1-\a}\, dx\right)^{1-\a}
\\
&= C \left(\irn V(x)|u|^{2}\, dx\right)^\a
\left(\irn|u|^\frac{\tau-2\a}{1-\a}\, dx\right)^{1-\a}
\\
&= C\|u\|_{V,\tau}^{2\a}\|u\|_{\frac{\tau-2\a}{1-\a}}^{\tau-2\a}
\\
&\le C\|u\|_V^{2\a}\|u\|_{\frac{\tau-2\a}{1-\a}}^{\tau-2\a}.
\end{align*}
Since  $H^1_V(\RN)$ is compactly embedded in $L^{\frac{\tau-2\a}{1-\a}}(\RN)$, the conclusion follows.
\end{proof}

We remark that a related result can be found in \cite{VSX}.
In the previous theorem $W$ could be diverging at infinity or not but, anyway, the case $V=W$ is excluded. This is, actually, the most intriguing and difficult situation, as the
following result suggests.

%
%

\begin{theorem}\label{Vnon}
Let $N\ge 3$ and $V \in \mathscr{C}^*(\RN)$. Suppose that there exist positive constants $c_1$, $c_2$ and $m$, and a  sequence $\{x_n\}_n$ in $\RN$ such that $V(x_n)\to +\infty$ and
\begin{equation*}
c_1V(x_n)\le V(x)\le c_2V(x_n), 
\quad\text{ for all } \ x\in B_{\frac{m}{\sqrt{V(x_n)}}}(x_n).
\end{equation*}
Then $H_V^1(\RN)\setminus L_V^\tau(\RN)$ is non-empty for all $\tau>2$. 
\end{theorem}

\begin{proof}
In the following, for $n\ge 1$, we denote $\mathcal{B}'_n:=B_{\frac{m}{2\sqrt{V(x_n)}}}(x_n)$, $\mathcal{B}''_n:=B_{\frac{m}{\sqrt{V(x_n)}}}(x_n)$  and $\mathcal{A}_n:=\mathcal{B}_n''\setminus \mathcal{B}_{n}'$. Since $V \in \mathscr{C}^*(\RN)$, we may assume that $\vert x_n \vert \to +\infty$.

For any $n\ge 1$, let $u_n\in C^1\big(\RN,[0,V(x_n)]\big)$ be such that \[
u_n(x):=
\begin{cases}
	\big(V(x_n)\big)^{\frac{N-2}4} &\text{for }x\in \mathcal{B}_{n}',
	\\
	0 &\text{for }x\in \RN\setminus \mathcal{B}_{n}'',
\end{cases}
\]
and  with $|\n u_n(x)|\le 2\big(V(x_n)\big)^{\frac N4}$, for  $x\in\mathcal{A}_{n}$.
\\
We have that 
\begin{align*}
\irn |\n u_n|^2\, dx
&=\int_{\mathcal{A}_n}|\n u_n|^2\,dx
\le c \big(V(x_n)\big)^{\frac N2} \meas(\mathcal{A}_n)\le c,
\\
\irn V(x)|u_n|^2\,dx
&\le \int_{\mathcal{B}_n''}V(x)u_n^2\,dx
\le c_2 V(x_n)\big(V(x_n)\big)^{\frac{N-2}2} \meas(\mathcal{B}_n'')\le c,
\end{align*}
and so $\{u_n\}_n$ is a bounded sequence in $H_V^1(\RN)$, while,
for any $\tau>2$,
\[
\irn V(x)|u_n|^\tau\,dx\ge  c_1V(x_n)\big(V(x_n)\big)^{\frac{(N-2)\tau}4} \meas(\mathcal{B}_n')\to +\infty, \qquad
\text{as }n \to +\infty,
\]
namely $\{u_n\}_n$ is unbounded in $L_V^\tau(\RN)$.\\
Now observe that, up to subsequence, we can assume that the balls $\mathcal B''_n$ are pairwise disjoint, and that $\big(V(x_n)\big)^{\frac{(N-2)(\tau-2)}{4\tau}}\ge 2^n.$  
If we define 
\begin{displaymath}
	v_n(x):=\frac{u_n}{\big(V(x_n)\big)^{\frac{(N-2)(\tau-2)}{4\tau}}},
\end{displaymath}
there exist positive constants $C_1$ and $C_2$  such that, for all $n\ge 1$,
\begin{displaymath}
	\|v_n\|_V\le \frac{C_1}{V(x_n)^{\frac{(N-2)(\tau-2)}{4\tau}}}\le\frac{C_1}{2^n}
\end{displaymath}
 and $\|v_n\|_{V,\tau}\ge C_2$. The function $w=\sum_{n=1}^{\infty}v_n$ belongs to $H_V^1(\RN)\setminus L_V^\tau(\RN)$.

\end{proof}

\begin{remark}
Observe that any positive  uniformly continuous and coercive potential satisfies all the conditions of Theorem \ref{Vnon}.
\end{remark}

It seems to be more challenging to prove a \emph{positive} statement, that is a continuous embedding result under reasonable assumptions on $V$. To the best of our knowledge this situation   has been studied in \cite{ADP} only in dimension \(N=2\). More precisely, the following continuous embedding theorem is proved in \cite{ADP}.

\begin{theorem}[\cite{ADP}]\label{th:emb}
  Let $V \in \mathscr{C}^*(\R^2)$.  If
the distributional derivatives of $V$ are functions satisfying
\begin{equation}\label{gradv}\tag{$\mathcal{V}$}
  |\n V(x)|\le C V^{\frac 32}(x)\quad \hbox{ for a.e. $x \in \RD$}
\end{equation}   
for some constant~$C>0$, then the space $H^1_V(\RD) $ is continuously
embedded into $L^\tau_V(\RD)$ for all $\tau\ge 2$.
\end{theorem}

\begin{remark}\label{re:VW}
If $V$ and $W$ are in $\mathscr{C}^*$, \eqref{gradv} is satisfied and $W(x)\leq C V(x)$ for all  $x\in\RD$ and for a suitable positive constant $C$, then $L_V^\tau(\RD)$ is continuously embedded into $L_W^\tau(\RD)$. So $H^1_V(\RD) $ is continuously embedded into $L^\tau_W(\RD)$ for all $\tau\ge 2$.
\end{remark}

It seems very hard to isolate, in the existing literature, precise conditions which ensure the continuous embedding of general weighted Sobolev spaces into weighted Lebesgue space without restrictive assumptions. An interesting statement appears in \cite[Theorem 2.4]{Avantaggiati}. Adapting the notation therein to ours, the author proves that if $\Omega$ satisfies the cone property in $\mathbb{R}^N$, then
\begin{displaymath}
    W^{1,p_0}_V(\Omega) \subset L^q_W(\Omega)
\end{displaymath}
continuously, provided that $1\leq p_0 <N$, $W \in C^1(\Omega) \cap L^\infty(\Omega)$ and $V$ are positive functions, and
\begin{displaymath}
    \left\vert \nabla W \right\vert \leq C V^{1/p_0}.
\end{displaymath}
The assumptions of this result exclude the case $p=N=2$. However, it witnesses that our assumption \eqref{gradv} has already been considered in some similar form.

Roughly speaking, assumption \eqref{gradv}  prevents the potential $V$ from
\emph{oscillating too much} at infinity. For example the potential
\begin{displaymath}
V(x)=|x|^2(\sin (e^{|x|})+2)+1
\end{displaymath}
does not satisfies \eqref{gradv}.

\bigskip

Now, assuming in the sequel the validity of the  continuous embedding $H^1_V(\RD) \hookrightarrow L^\tau_V(\RD),$
we investigate if such embedding is also compact.

As a first step we provide a general condition which assures the lack of the compact embedding, then we study some more specific cases.

\begin{proposition}\label{pr:general}
	Let $V\in \mathscr{C}^*(\RD)$  such that $\HV\hookrightarrow L^\tau_V(\RD)$ for some $\tau>2$. If $V$ satisfies 
\begin{enumerate}[label=$(\mathcal{V}_0)$,ref=$\mathcal{V}_0$]
\setcounter{enumi}{-1}
		\item\label{ipotesiV} there exist $c_1>0$ $,c_2>0$, $m>0$ and a  sequence $\{x_n\}_n \subset\RD$ with  $|x_n| \rightarrow +\infty$ and such that, for any $n\ge 1$,
			$$c_1 V(x_n)\le V(x)\le c_2 V(x_n), \quad\hbox{ in }B_{\frac {m}{\sqrt{V(x_n)}}}(x_n),$$
\end{enumerate}
	then $\HV$ is not compactly embedded into $L^\tau_V(\RD)$.
\end{proposition}

\begin{proof}
	In the following, for $n\ge 1$, we denote 
\begin{displaymath}
  \mathcal{B}'_n:=B_{\frac {m}{2\sqrt{V(x_n)}}}(x_n), \quad
  \mathcal{B}''_n:=B_{\frac {m}{\sqrt{V(x_n)}}}(x_n), \quad
  \mathcal{A}_n:=\mathcal{B}_n''\setminus \mathcal{B}_{n}'
\end{displaymath}
For any $n\ge1$, we consider a function $u_n\in C^1(\RD,[0,1])$  such that \(u_n=1\) in \(\mathcal{B}_{n}'\), \(u_n=0\) in \(\RD \setminus \mathcal{B}_n''\), and \(|\n u_n|\le (3/m) \sqrt{V(x_n)}
\) in \(\mathcal{A}_{n}\).

	By  assumption, we have that
	\begin{align*}
	&\ird |\n u_n|^2\, dx
	=\int_{\mathcal{A}_n}|\n u_n|^2\,dx
	\le \frac 9{m^2}V(x_n) \cdot \meas(\mathcal{A}_n)\simeq c>0,
	\\
	&\ird V(x)|u_n|^2\,dx
	\le \int_{\mathcal{B}_n''}V(x)\,dx
	\le\sup_{\mathcal{B}_n''}V\cdot \meas(\mathcal{B}_n'')\simeq c''>0,
	\\
	&\ird V(x)|u_n|^\tau\,dx
	\ge \int_{\mathcal{B}_n'}V(x)\,dx
	\ge\inf_{\mathcal{B}_n'}V \cdot \meas(\mathcal{B}_n')\simeq c'>0.
	\end{align*}
We conclude that $\{u_n\}_n$ is a bounded sequence in $\HV$ such that the support of $u_n$ is \emph{travelling} at infinity and then $u_n\weakto 0$ in $\HV$ but $\|u_n\|_{V, \tau}$ is bounded away from zero. This shows that $\HV$ is not compactly embedded into $L^\tau_V(\RD)$.
\end{proof}

Here we present two specific cases where $V$ satisfies \eqref{ipotesiV}.
\begin{corollary}
Let $V\in \mathscr{C}^*(\RD)$. If   $\HV\hookrightarrow L^\tau_V(\RD)$ for some $\tau>2$  and $V$ satisfies one of the following:
\begin{enumerate}[label=\arabic{*}.,ref=\arabic{*}]
\setcounter{enumi}{0}
\item\label{caso1} $V$ is  uniformly continuous;
\item\label{caso2} there exists $m>0$ such that $V$ satisfies the following asymptotic locally  Lipschitz-type property:
\begin{equation}\label{supershit}\tag{$\mathcal{V}_1$}
\sup_{\substack{y \neq x \\ |y-x|\le\frac {m}{\sqrt{V(y)}}}}
	\frac{|V(y)-V(x)|}{|y-x|}=o\big(V^{\frac 32}(y)\big)\quad \hbox{ for } |y|\to +\infty;
\end{equation}
\end{enumerate}
then the embedding of $\HV$ in $L^\tau_V(\RD)$ is not compact.
\end{corollary}

\begin{proof}
{\sc Case }\ref{caso1}.
Let $\a>0$ be such that $V\ge \a$ in $\RD$. Since $V$ is uniformly continuous,  there exists $\d>0$ such that, if $|x-y|<\d$, we have 
	$|V(x)-V(y)|< \a/2$. Let  $\{x_n\}_n$ be  a divergent sequence. We set $m:=\d \sqrt{\a}$. For any $n\ge 1$, if $x\in B_{\frac {m}{\sqrt{V(x_n)}}}(x_n)$, we have $|x-x_n|<\d$ and so
	$$V(x)\le V(x_n)+|V(x)-V(x_n)|\le V(x_n)+\frac \a 2\le \frac 32 V(x_n)$$
	and 	$$\frac 12 V(x_n)\le V(x_n)-\frac \a 2\le V(x_n)- |V(x)-V(x_n)|\le V(x)$$
	and then property \eqref{ipotesiV} holds.
\vskip .2cm
{\sc Case }\ref{caso2}. Since \eqref{supershit} holds, there exists $y_0\in \RD$ with $|y_0|$ large enough such that, for any $|y|>|y_0|$ and $x\in B_{\frac {m}{\sqrt{V(y)}}}(y)$,
$$|V(y)-V(x)|\le m' \big(V(y)\big)^{\frac 32}|x-y|$$
where $m'\in (0,1/m)$.
Take any divergent sequence $\{x_n\}_n$ such that $|x_n|>|y_0|$, for all $n\ge 1.$ For any $x\in B_{\frac {m}{\sqrt{V(x_n)}}}(x_n)$, we have
\begin{align*}
V(x)&\le V(x_n)+|V(x)-V(x_n)|\le V(x_n) +m' \big(V(x_n)\big)^{\frac 32}|x-x_n|\le (1+mm') V(x_n)\\
V(x)&\ge V(x_n)-|V(x)-V(x_n)|\ge V(x_n) -m' \big(V(x_n)\big)^{\frac 32}|x-x_n|\ge (1-mm') V(x_n)
\end{align*}
and then property  \eqref{ipotesiV} holds.

\end{proof}
\begin{remark}
Even if  assumption \eqref{supershit} seems to be quite technical, there exists a large class of potentials satisfying it. In particular observe that, if we assume that $V$  is coercive and $V\in C^1(\RD\setminus B_R)$, for some $R>0$, we just have to verify the simpler condition

		\begin{equation}\label{megashit}\tag{$\mathcal{V}_2$}
		\exists\, \eps>0 \hbox{ such that }\sup_{z\in B_\eps (y)}|\n V(z)|=o\big(V^{\frac 32}(y)\big), \quad \hbox{ for } |y|\to +\infty.
		\end{equation}
		Indeed, by coercivity, for an arbitrary $m>0$ there exists $y_0\in \RD$ with $|y_0|$ large enough such that for any $|y|>|y_0|$ we have $\frac m{\sqrt{V(y)}}<\eps$. Assuming \eqref{megashit}, as a consequence of multidimensional mean value theorem we obtain \eqref{supershit} as follows
			\begin{equation*}
			\sup_{\substack{y \neq x \\ |y-x|\le\frac {m}{\sqrt{V(y)}}}}
				\left|\frac{V(y)-V(x)}{|y-x|}\right|\le \sup_{z\in B_\eps (y)}|\n V(z)|=o\big(V^{\frac 32}(y)\big),\quad \hbox{ for } |y|\to +\infty.
			\end{equation*} 
		Condition \eqref{megashit} is satisfied for example by all potentials of the type $|x|^\a$ with $\a>0$ but also by potentials growing exponentially fast likes $\exp \left( \vert x \vert^\alpha \right)$ with $\a>0$.

\end{remark}

We conclude with another counterexample where we assume neither regularity nor coercivity assumptions on the potential.	

\begin{proposition}\label{pr:noncomp}
If $V\colon \RD\to \R$ is bounded below by a positive constant and is a  measurable function such that $V(x)=n^2$, for all $n\ge 1$ and for all $x\in \RD$ such that $n-\frac 1n \le |x|\le n+\frac 1n$, then $\HV$ is not compactly embedded into $L^\tau_V(\RD)$, for $\tau>2$.
\end{proposition}

\begin{proof}
In the following, for $n\ge 1$, we denote $\mathcal{B}'_n:=B_{\frac 1{2n}}(n,0)$, $\mathcal{B}''_n:=B_{\frac 1n}(n,0)$  and $\mathcal{A}_n:=\mathcal{B}_n''\setminus \mathcal{B}_{n}'$.
For any $n\ge1$, pick a function $u_n\in C(\RD,[0,1])$  such that
\[
u_n(x):=
\begin{cases}
1 &\text{for }x\in \mathcal{B}_{n}',
\\
0 &\text{for }x\in \RD\setminus \mathcal{B}_{n}'',
\end{cases}
\]
and  with $|\n u_n(x)|\le 2n$, for  $x\in\mathcal{A}_{n}$.
\\
We have that, for any $\tau>2$,
\begin{align*}
\ird |\n u_n|^2\, dx
&=\int_{\mathcal{A}_n}|\n u_n|^2\,dx
=4n^2 \meas(\mathcal{A}_n)\simeq c>0,
\\
\ird V(x)|u_n|^2\,dx
&\le \int_{\mathcal{B}_n''}V(x)\,dx
=n^2 \meas(\mathcal{B}_n'')\simeq c>0,
\\
\ird V(x)|u_n|^\tau\,dx
&\ge \int_{\mathcal{B}_n'}V(x)\,dx
=n^2 \meas(\mathcal{B}_n')\simeq c'>0.
\end{align*}
We conclude as in Proposition \ref{pr:general}.
\end{proof}

Finally,  we deal with the last remaining case, namely the one dimensional case.
\begin{theorem}\label{th:main1}
 Let $V \in \mathscr{C}^*(\R)$ be coercive.  If
the distributional derivative of $V$ is a function satisfying
\begin{equation}\label{gradv1}\tag{$\mathcal{V'}$}
  |V'(x)|\le C V^{\frac 32}(x)\quad \hbox{ for a.e. $x \in \R$}
\end{equation}   
for some constant~$C>0$, then the space $H^1_V(\R) $ is compactly
embedded into $L^\infty(\R)$ and into $L^\tau_V(\R)$ for all $\tau>2$.
\end{theorem}

\begin{proof}
First we prove the continuous embeddings. 
Let $u\in C^1_c(\R)$, namely of class $C^1$ and with compact support, and let $w=\sqrt{V}u^2$. Observe that, for any $x\in \R$ we have 
\begin{equation} \label{eq:4}
w(x)=\int_{-\infty}^{x}w'(t)\,dt=\int_{-\infty}^{x}\left(\frac{V'(t)}{2 \sqrt{V(t)}}u^2(t)+2 \sqrt{V(t)}u(t)u'(t)\right)dt.
\end{equation}
So, for any $x\in \R$, equation~\eqref{eq:4} yields
\begin{align*}
u^2(x)&\le C w(x)
\le \int_{\R}|w'|\, dx
\le C\int_{\R}\left(\frac{|V'(x)|}{ \sqrt{V(x)}}u^2(x)+ \sqrt{V(x)}|u(x)||u'(x)|\right)dx \nonumber
\\
&\le C \left(\|u\|_{V,2}^2+\|u\|_{V,2}\|u'\|_2\right)
\le C \|u\|_V^2.
\end{align*}
This implies that
\begin{equation}\label{key}
\|u\|_\infty \le C \|u\|_V, \qquad\text{for all }u\in C^1_c(\R).
\end{equation}
The embedding of $H_V^1(\R)$ into $L^\infty(\R)$ follows by density.

Let now $\tau>2$ and $u\in H_V^1(\R)$. We have 
\begin{equation} \label{eq:8}
\|u\|_{V,\tau}^\tau
\le \|u\|_\infty^{\tau-2}\|u\|_{V,2}^2
\le C\|u\|_V^2,
\end{equation}
and the conclusion about the continuous embeddings  follows easily.

\medskip

Let us now show the compact embeddings. Let us go back to equation \eqref{eq:4}. Fix $\varepsilon >0$, and choose $R>0$ so large that $1/V(x)<\varepsilon^4$ for each $x \in \mathbb{R} \setminus (-R,R)$. It is known (see \cite[Théorème VIII.7]{brezis}) that $H^1(-R,R)$ is compactly embedded into $L^\infty(-R,R)$. On the other hand, if $\vert x \vert \geq R$ then \eqref{eq:4} implies
\begin{equation*}
	\sqrt{V(x)} u^2(x) \leq C \Vert u \Vert_V^2,
\end{equation*}
and thus
\begin{equation} \label{eq:7}
	\vert u(x) \vert \leq \left( \sqrt{C} \Vert u \Vert_V \right)  \varepsilon.
\end{equation}
It now follows from \eqref{eq:7} that $H_V^1(\mathbb{R}\setminus (-R,R))$ is compactly embedded into $L^\infty(\mathbb{R} \setminus (-R,R))$. We conclude that $H_V^1(\mathbb{R})$ is compactly embedded into $L^\infty(\mathbb{R})$, and also in $L_V^\tau(\mathbb{R})$ by \eqref{eq:8}, for all $\tau>2$.
\end{proof}

\begin{remark}\label{re:VW1}
If $V$ and $W$ are two potentials from $\mathscr{C}^*(\R)$ such that $W(x)\leq C V(x)$ for all  $x\in\R$ and for a suitable positive constant $C$, then $L_V^\tau(\R)$ is continuously embedded into $L_W^\tau(\R)$. Therefore, if $V \in \mathscr{C}^*(\R)$ is coercive and satisfies \eqref{gradv1} then $H^1_V(\R) $ is compactly embedded into $L^\tau_W(\R)$ for all $\tau> 2$. 
\end{remark}

\subsection{The radial setting}

\

It is well-known that, for $N\ge 2$, the constraint of radial symmetry improves the compactness properties of Sobolev spaces, see \cite{willem}.
Here we  study the embedding of $H^1_{V,{\rm rad}}(\RN)$ into $L^\tau_{W,{\rm rad}}(\RN)$, where
$H^1_{V,{\rm rad}}(\RN)$ and $L^\tau_{W,{\rm rad}}(\RN)$ are, respectively, the subsets of radial functions contained in  $H^1_{V}(\RN)$ and $L^\tau_{W}(\RN)$.

\begin{theorem}\label{thrad}
	Let $N\ge 2$, $V$ and $W$ be two potentials in $\mathscr{C}^*(\RN)$ such that there exist $\phi:\RN \to R_+$ vanishing at infinity, $\bar \tau>2$, and $\tilde R>0$ with
	\begin{equation}\label{VWx}
    \tag{$\mathcal{VW}_{{\rm rad}}$}
	W(x)\le C \phi(x) V(x)|x|^{\frac{(N-1)(\bar\tau-2)}{2}},
	\quad\text{ whenever }|x|\ge \tilde R
	\end{equation}
	Then the following compact embeddings hold 
	$$H^1_{V,{\rm rad}}(\RN)\hookrightarrow \hookrightarrow L^\tau_{W,{\rm rad}}(\RN),
	\qquad\text{ for all }
	\begin{cases}
	\tau \ge  \bar \tau & \text{if }N=2,
	\\
	\bar \tau\le \tau < 2^* & \text{if }N\ge 3.
	\end{cases}$$ 
\end{theorem}

\begin{proof}
	Let $\{u_n\}_n$ be a bounded sequence of $H^1_{V,{\rm rad}}(\RN)$. There exists $u\in H^1_{V,{\rm rad}}(\RN)$ such that $u_n \weakto u$ weakly in $H^1_{V,{\rm rad}}(\RN)$. By linearity, we can assume that $u=0$.
	\\
	Let $R>1$. Since $V$ is bounded below by a positive constant, $\{u_n\}_n$ is a bounded sequence of $H^1_{\rm rad}(\RN)$ as well. So, by Strauss Radial Lemma (see \cite{strauss}) and \eqref{VWx},  we have for any $R\ge\tilde R$,
	\begin{align*}
	\int_{B_R^c} W(x)|u_n|^\tau\, dx 
	&=\int_{B_R^c} W(x)|u_n|^{\tau-2}u_n^2\, dx 
	\\
	&\le C\int_{B_R^c} \frac{W(x)}{|x|^{\frac{(N-1)(\tau-2)}{2}}}u_n^2\, dx
	\\
	&\le C\int_{B_R^c} \frac{\phi(x) V(x)|x|^{\frac{(N-1)(\bar\tau-2)}{2}}}{|x|^{\frac{(N-1)(\tau-2)}{2}}}u_n^2\, dx
	\\
	&\le C\sup_{B_R^c}\phi\int_{B_R^c} V(x)u_n^2\, dx
	\le  C\sup_{B_R^c}\phi.
	\end{align*}
	Therefore, for any $\eps>0$ we can find $R>0$ sufficiently large such that 
	\[
	\int_{B_R^c} W(x)|u_n|^\tau\, dx <\frac \eps 2 \qquad \text{for all }n\ge 1.
	\]
	On the other hand, since $V$ and $W$  are locally bounded and $H^1(\RN)$ is locally compact embedded into $L^\tau(\RN)$, we have that 
	\[
	\int_{B_R} W(x)|u_n|^\tau\, dx <\frac \eps 2 \qquad \text{for all $n$ sufficiently large},
	\]
	and we conclude that  
	\[
	\irn W(x)|u_n|^\tau\, dx \to 0 \qquad \text{as $n \to +\infty$}.
	\]
\end{proof}

\begin{remark}
    Observe that under condition \eqref{VWx},   $\liminf_{|x|\to +\infty}\frac{V(x)}{W(x)}$ may be   $+\infty$ or  any non-negative number and so also zero. In particular $V$ and  $W$ are allowed to be respectively bounded and coercive provided that 
        \begin{displaymath}
            \limsup_{|x|\to+\infty}\frac{W(x)}{\phi(x)|x|^{\frac{ (N-1)(\bar \tau-2)}{2}}}<+\infty.
        \end{displaymath}
\end{remark}

\begin{remark}
	The class $\mathscr{C}^*(\RN)$ is made of \emph{continuous} functions. This is rather natural requirement for applications to PDEs. It should be noticed, however, that our assumptions on the potential functions $V$ and $W$ may be slightly relaxed. Continuity may be replaced by measurability in most statements, but of course the assumption $V \in L^1_{\mathrm{loc}}$ is needed whenever distributional derivatives of $V$ are taken into account.
\end{remark}

\section{Applications to a nonlinear eigenvalue problem}\label{se:exi}

This section is devoted to the study of a planar nonlinear Schr\"odinger equation in presence of the same coercive potential in the linear term and in the nonlinearity. By what observed in Section \ref{se:func}, the fact that $\HV$ is continuously but not compactly embedded into $L_V^{p+1}(\RD)$ makes the problem quite tough.

\begin{theorem} \label{th:main}
Let $p\in (1,+\infty)$ and assume that $V\colon \RD\to\R$ is a $C^1$ positive coercive potential satisfying \eqref{gradv}.
Then there exists a nontrivial pair $(\bar u,\lambda)\in \big(\HV\cap C^2(\RD)\big)\times\R_+$, $u\ge 0$,  solving the problem \eqref{eq:11}.
Explicitly, 	
\begin{displaymath}
	\l=\frac2{p+3}\frac{\|\bar u\|_{V}^{2}}{\|\bar u\|_2^2}.
\end{displaymath}
\end{theorem}


To achieve the proof, we define $I\colon \HV \to \R$ by
	\begin{equation}\label{I}
		I(u):=\frac 12 \ird \left(|\n u|^2 + V(x)u^2\right)dx -\frac 1{p+1}\|u\|_2^2\ird V(x)|u|^{p+1}\, dx,
	\end{equation}
for all $u\in \HV$, and
\begin{displaymath}
	\Ne:= \left\{ u\in H^1_V(\RD)\mid u\neq 0, \; J(u)=0 \right\},
\end{displaymath}
where $J\colon \HV \to \R$  is such that
\begin{displaymath}
	J(u):=I'(u)[u]=\ird \left(|\n u|^2 + V(x)u^2\right) dx -\frac {p+3}{p+1}\|u\|_2^2\ird V(x)|u|^{p+1}\, dx,
\end{displaymath}
for all $u\in \HV$.
\begin{lemma}\label{le:ne}
$\Ne$ is a natural constraint.
\end{lemma}

\begin{proof}
The conclusion follows observing that
\[
	J'(u)[u]= -(p+3)\|u\|_2^2\ird V(x)|u|^{p+1}\, dx<0,
\]
for all $u\in \Ne$.
\end{proof}

\begin{lemma}\label{le:p+1}
There exists $c>0$ such that $\|u\|_2\cdot\|u\|_{V,p+1}\ge c$, for any $u\in \Ne$.
\end{lemma}

\begin{proof}
Since $\HV$ is continuously embedded into $L_V^{p+1}(\RD)$, we have
\[
\|u\|_{V,p+1}^2\le C\|u\|_V^2=C'\|u\|_{2}^2\cdot \|u\|_{V,p+1}^{p+1},
\]
and we conclude.
\end{proof}

We set 
\begin{displaymath}
	\s=\inf_{u\in\Ne}I(u).
\end{displaymath}
\begin{lemma}\label{le:sigma}
We have that $\s>0$.
\end{lemma}
\begin{proof}
Since, 
\begin{equation}\label{I|N}
I(u)=\frac{p+1}{2(p+3)}\ird \left(|\n u|^2 + V(x)u^2\right) dx
=\frac 12\|u\|_2^2\ird V(x)|u|^{p+1}\, dx \quad\text{for any }u \in \Ne, 
\end{equation}
by Lemma \ref{le:p+1} we deduce that $\s>0$.
\end{proof}

\begin{lemma}\label{le:u_0}
Let $\{u_n\}_n$ be a sequence in $\Ne$ such that $I(u_n)\to \s$,  there exists $u_0\in \HV\setminus\{0\}$ such that $u_n\rightharpoonup u_0$ in $\HV$. 
\end{lemma}
\begin{proof}
By \eqref{I|N}, it is easy to see that $\{u_n\}_n$ is bounded in $\HV$ and then, up to a subsequence, there exists $u_0\in \HV$ such that $u_n\rightharpoonup u_0$ in $\HV$.  
\\
Since $V$ is coercive, $\HV$ is compactly embedded into $L^2(\RD)$ and so
	\begin{equation}\label{eq:str}
		u_n\to u_0 \hbox{ in }L^2({\RD}).
	\end{equation}
Since $\{u_n\}_n$ is bounded in $\HV$ and hence also in $L^{p+1}_V(\RD)$,  by Lemma \ref{le:p+1}, we deduce that $u_0\neq 0$.
\end{proof}

For any $n\ge 1$ and an arbitrary $\alpha \in \left(\frac 1{p+3},\frac 12\right)$ we define the positive measure \(\nu_n\) by
\begin{equation*}
\nu_n(\O):= 
\left(\frac 12 -\a\right)\int_\O\left(|\n u_n|^2 + V(x)u_n^2\right) dx+\frac{\alpha(p+3)-1}{p+1}\|u_n\|_2^2\int_\O V(x)|u_n|^{p+1}\, dx
\end{equation*}
for each measurable subset~$\O\subset \R^2$,  and
\begin{equation*}
G(u):= \left(\frac 12 -\a\right)\ird\left(|\n u|^2 + V(x)u^2\right) dx+\frac{\alpha(p+3)-1}{p+1}\|u\|_2^2\ird V(x)|u|^{p+1}\, dx
\end{equation*}
for any $u\in H^1_V(\RD)$.
Of course $\nu_n(\RD)=G(u_n)=I(u_n)=\s+o_n(1)$.

By \cite{Lions1,Lions2} there are three possibilities:
\begin{itemize}
	\item[1.] {\it concentration:} there exists a sequence $\{\xi_n\}_n$ in $\RD$ with the following property: for any $\epsilon> 0$, there exists $r = r(\epsilon) > 0$ such that
	\[
	\nu_n(B_r(\xi_n))\ge \s-\epsilon;
	\]
	\item[2.] {\it vanishing:} for all $r > 0$ we have that
	\[
	\lim_{n \to +\infty} \sup_{\xi\in\RD} \nu_n(B_r(\xi))=0;
	\]
	\item[3.] {\it dichotomy:} there exist two sequences of positive measures $\{\nu_n^1\}_n$ and $\{\nu_n^2\}_n$, a positively diverging sequence of numbers $\{R_n\}_n,$ and $\tilde \s \in (0,\s)$ such that
	\begin{align*}
	&0\le \nu_n^1 + \nu_n^2\le \nu_n,\quad \nu_n^1(\RD)\to \tilde \s,\quad \nu_n^2(\RD)\to \s-\tilde \s, \\
	& \operatorname{supp} \nu_n^1\subset B_{R_n},\quad \operatorname{supp} \nu_n^2\subset B_{2R_n}^c.
	\end{align*}			
\end{itemize}
We are going to rule out both vanishing and dychotomy.

In order to exclude the vanishing, we need the equivalent of  \cite[Lemma I.1]{Lions2} in $\HV$.
\begin{lemma}\label{le:lions}
Suppose that $V$ satisfies \eqref{gradv}. Let $\{u_n\}_n$ be a bounded sequence in $\HV$ such that
\begin{equation}\label{lions}
\lim_{n \to +\infty} \sup_{\xi\in\RD} \int_{B_r(\xi)} V(x) u_n^2\, dx=0,
\end{equation}
for some $r>0$, then $u_n \to 0$ strongly in $L_V^\tau(\RD)$, for all $\tau>2$, as $n \to +\infty$.
\end{lemma}

\begin{proof}
We start with the following intermediate step.

\vspace{.2cm}
{\sc Claim:}  $u_n \to 0$ strongly in $L_V^4(\RD)$,  as $n \to +\infty$.
\\
Since $W^{1,1}(B_r(\xi))\hookrightarrow L^2(B_r(\xi))$ and by \eqref{gradv}, we have
\begin{align*}
\int_{B_r(\xi)} V(x) u_n^4\, dx
&=\|\sqrt{V}u_n^2\|_{L^2(B_r(\xi))}^2
\le C\|\sqrt{V}u_n^2\|_{W^{1,1}(B_r(\xi))}^2
\\
&\le C\left(\int_{B_r(\xi)}  
\Big(\frac{|\n V(x)|}{2\sqrt{V(x)}} u_n^2 
+2\sqrt{V(x)} u_n |\n u_n|
+\sqrt{V(x)} u_n^2\Big) dx\right)^2
\\
&\le C\left(\int_{B_r(\xi)}  
\Big(V(x) u_n^2+\sqrt{V(x)} u_n |\n u_n|\Big) dx\right)^2
\\
&\le C\left(\|u_n\|_{L_V^2(B_r(\xi))}^2 
+\|u_n\|_{L_V^2(B_r(\xi))}\|\n u_n\|_{L^2(B_r(\xi))}\right)^2
\\
&\le C\|u_n\|_{L_V^2(B_r(\xi))}^2 
\left(\|u_n\|_{L_V^2(B_r(\xi))}+\|\n u_n\|_{L^2(B_r(\xi))}\right)^2
\\
&\le C\|u_n\|_{L_V^2(B_r(\xi))}^2 
\|u_n\|_{H_V^1(B_r(\xi))}^2
\\
&= o_n(1)
\|u_n\|_{H_V^1(B_r(\xi))}^2.
\end{align*}
Then, covering $\RD$ by ball of radius $r$ in such a way that any point of $\RD$ is contained in at most $m$ balls (where $m$ is a prescribed integer), being $\{u_n\}_n$ bounded in $\HV$, we deduce
\[
\|u_n\|_{V,4}^2\le m \, o_n(1)\|u_n\|_V^2=o_n(1).
\]

\vspace{.2cm}
{\sc Claim:}  $u_n \to 0$ strongly in $L_V^\tau(\RD)$,  for any $\tau>2$, as $n \to +\infty$.
\\
The conclusion follows by an interpolation argument by \cite[Lemma 2.1]{ADP}.
\end{proof}

\begin{lemma}\label{le:van}
	Vanishing does not hold.
\end{lemma}
\begin{proof}
	If vanishing holds, we  have that
	\[
	\lim_{n \to +\infty} \sup_{\xi\in\RD} \int_{B_r(\xi)} V(x) u_n^2\, dx=0
	\]
and so, by Lemma \ref{le:lions}, 
we deduce that ${u_n}\to 0 $ strongly in $L^{p+1}_V(\RD)$ and then, being $J(u_n)=0$, we would have  $u_n\to 0$ in $\HV$, contradicting $I(u_n)\to \s>0$.
\end{proof}

\begin{lemma}\label{le:dicho}
	Dichotomy does not hold.
\end{lemma}
\begin{proof}
	As usual, we perform a proof by contradiction assuming that, on the contrary, dichotomy holds.\\
	Consider a radial function $\rho_n\in C^1_0(\RD,[0,1])$  such that, for any $n\ge 1$,  $\rho_n\equiv 1$ in $B_{R_n}$, $\rho_n\equiv 0$ in $ B_{2R_n}^c$ and $\sup_{x\in\RD}|\n\rho_n(x)|\le 2/R_n$. Moreover set $v_n=\rho_nu_n$ and $w_n=(1-\rho_n)u_n$. It is clear  that $v_n, w_n\in H^1_V$.
	Now we proceed by steps.
	
\vspace{.2cm}
	{\sc First step}: we claim that
	\begin{align}\label{eq:onu}
	&\left(\frac 12 -\a\right)\int_{\O_n}(|\n u_n|^2+V(x)u_n^2)\, dx+\frac{\alpha(p+3)-1}{p+1}\|u_n\|_2^2\int_{\O_n} V(x)|u_n|^{p+1}\, dx \to 0, \\\label{eq:onv}
	&\left(\frac 12 -\a\right)\int_{\O_n}(|\n v_n|^2+V(x)v_n^2)\, dx+\frac{\alpha(p+3)-1}{p+1}\|v_n\|_2^2\int_{\O_n} V(x)|v_n|^{p+1}\, dx \to 0, \\\label{eq:onw}
	&\left(\frac 12 -\a\right)\int_{\O_n}(|\n w_n|^2+V(x)w_n^2)\, dx+\frac{\alpha(p+3)-1}{p+1}\|w_n\|_2^2\int_{\O_n} V(x)|w_n|^{p+1}\, dx \to 0, 
	\end{align}
	as $n \to +\infty$, where $\O_n:=\{x\in\RD: R_n\le |x|\le 2R_n\}$.

	Let's prove \eqref{eq:onu}. Observe that 
	\begin{align*}
	\nu_n(\O_n)&=\s-\nu_n(B_{R_n})-\nu_n(B^c_{2R_n})+o_n(1)\\
	&\le \s- \nu^1_n(B_{R_n})-\nu^2_n(B^c_{2R_n})+o_n(1)=o_n(1)
	\end{align*}
	and then we deduce \eqref{eq:onu}.
	Moreover, by simple computations
	\begin{align*}
	&\left(\frac 12 -\a\right)\int_{\O_n} \left(|\n v_n|^2+V(x)v_n^2 \right)\, dx+\frac{\alpha(p+3)-1}{p+1}\|v_n\|_2^2\int_{\O_n} V(x)|v_n|^{p+1}\, dx\\
	&\qquad\le 2\left(\frac 12 -\a\right) \int_{\O_n}\left(|\n u_n|^2+\left(V(x)+\frac4{R^2_n}\right) u_n^2\right)\, dx \\
	&\qquad\qquad {}+\frac{\alpha(p+3)-1}{p+1}\|u_n\|_2^2\int_{\O_n} V(x)|u_n|^{p+1}\, dx \\
	&\qquad\le o_n(1)
	\end{align*}
	and then we have proved also \eqref{eq:onv}. The proof of \eqref{eq:onw} is analogous.
	
\vspace{.2cm}
	{\sc Second step}: $\liminf_{n \to +\infty}  G(v_n)=\tilde{\s}$.\\
	Observe that
	\begin{equation}\label{superfico}
	G(v_n)\ge \nu_n (B_{R_n})\ge \nu_n^1 (B_{R_n})\to \tilde{\s},
	\end{equation} 
	Now, observe that, by the first step and considering that $\nu_n\ge\nu_n^2$,
	\begin{align*}
	\s&=\lim_{n \to +\infty}\nu_n(\RD)=\lim_{n \to +\infty} (\nu_n(B_{R_n})+\nu_n(B_{2R_n}^c))\\
	&\ge \liminf_{n \to +\infty}  G(v_n)+\lim_{n \to +\infty}\nu_n^2(B_{2R_n}^c).
	\end{align*}
	Since $\lim_{n \to +\infty}\nu_n^2(\RD)=\s-\tilde \s$ and $\supp \nu_n^2\subset B_{2R_n}^c$, we conclude that 
	$$\liminf_{n \to +\infty}  G(v_n)=\tilde \s.$$
	
\vspace{.2cm}
	{\sc Third step}: conclusion.\\
	First of all observe that, since $u_n=v_n+w_n$, then by the first step
	\begin{equation}\label{eq:Gun}
	G(u_n)\ge G(v_n)+G(w_n)+o_n(1). 
	\end{equation}
	Observe that, by first step, 
	\begin{equation}\label{eq:Jun}
	0=J(u_n)\ge J(v_n)+J(w_n)+o_n(1).
	\end{equation}
	For any $n\in \N$, let $t_n, s_n>0$ be the numbers, respectively, such that $t_nv_n\in \Ne$ and  $s_nw_n\in \Ne$.
	
	There are now three possibilities.
	
	\textit{Case 1}:  up to a subsequence, $J(v_n)\le 0$. 
	\\
	By simple computations we see that $t_n\le 1$ and then we have
	\begin{equation*}
	\s\le I(t_nv_n)=G(t_nv_n)\le G(v_n),
	\end{equation*}
	which, for a large $n\ge 1$, leads to a contradiction due to the fact that, by the second step,
	$$\s>\tilde \s= \liminf_{n \to +\infty} G(v_n).$$
	
	{\it Case 2}:  up to a subsequence, $J(w_n)\le 0.$\\
	Then, proceeding as in the first case, by \eqref{superfico} and using \eqref{eq:Gun}, we have, for $n$ sufficiently large,
	\begin{equation*}
	\s\le I(s_nw_n)=G(s_nw_n)\le G(w_n)\le G(u_n)+o_n(1)=\s+o_n(1),
	\end{equation*}
	that is $\s=\lim_{n \to +\infty} G(w_n)$. Then, passing to the limit in \eqref{eq:Gun}, we have
	$$\s\ge \s + \liminf_{n \to +\infty} G(v_n)$$ 
	which contradicts the result obtained in the second step.
	
	{\it Case 3}:  there exists $n_0\ge 1$ such that for all $n\ge n_0$ both $J(v_n)>0$ and $J(w_n)>0$. 
	\\
	Then $\liminf_{n \to +\infty} t_n\ge1$ and, by \eqref{eq:Jun}, we also have that $J(v_n)=o_n(1)$. \\
	If $ \liminf_{n \to +\infty}  t_n = 1$, we can repeat the computations performed in the first case and get the contradiction. If $\liminf_{n \to +\infty}  t_n >1$, from
	\begin{equation*}
	o_n(1) = J(v_n)-\frac 1{t_n^{p+3}}J(t_nv_n)=\left(1-\frac 1{t_n^{p+1}}\right)\| v_n\|^2_V,
	\end{equation*}
	we get a contradiction since $\liminf_{n \to +\infty}  G(v_n)>0$ by the second step.
\end{proof}

\begin{proposition}\label{conc}
	Concentration holds and, moreover, the sequence $\{\xi_n\}_n$ is bounded.
\end{proposition}
\begin{proof}
By Lemmas \ref{le:van} and \ref{le:dicho}, we deduce that concentration occurs.

Now assume, by contradiction, that $\{\xi_n\}_n$ is unbounded.
Since, by Lemma \ref{le:u_0}, $u_0\neq 0$, there exists $ R>0$ such that
\[
m:=\frac{\alpha(p+3)-1}{p+1}\|u_0\|_2^2\int_{B_{R}} V(x)|u_0|^{p+1}\, dx>0.
\]
Since the embedding of $H^1_V(B_{ R})$ in $L^\tau_V(B_{ R})$ is compact, for all $\tau\ge 2$, we have 
		\begin{equation*}
			\liminf_{n \to +\infty}  \nu_n(B_{R})\ge \frac{\alpha(p+3)-1}{p+1}\|u_0\|_2^2\int_{B_{R}} V(x)|u_0|^{p+1}\, dx=m>0.
		\end{equation*}
	Consider $\eps=m/2$ and apply the concentration hypothesis. We would have for some $\tilde R>0$  that  $\nu_n(B_{\tilde R}(\xi_n))\ge \s-\varepsilon$ for all $n\ge 1$.
	\\
	Then, since for large $n\ge1$ we would have $B_{R}\cap B_{\tilde R}(\xi_n)=\emptyset$, the following should hold
		\begin{displaymath}
		\s=\lim_{n \to +\infty}\nu_n(\RD)\ge \liminf_{n \to +\infty} \nu_n(B_{R}) + \liminf_{n \to +\infty}  \nu_n(B_{\tilde R}(\xi_n)) \ge \s+\frac m2,
		\end{displaymath}
reaching a contradiction.
\end{proof}

\begin{proof}[Proof of Theorem \ref{th:main}]
By Proposition \ref{conc}, $u_n$ tends to $u_0$ in $H^1_V(\RD)$. As a consequence $u_0$ solves the minimizing problem 
\begin{displaymath}
	I(u_0)=\inf_{u\in \Ne}I(u).
\end{displaymath}
Of course, since also $|u_0|$ solves the same minimizing problem, we may assume $u_0$ nonnegative.
By Lemma \ref{le:ne} and standard arguments, it solves the problem
\begin{displaymath}
	-\Delta u_0 + \left(V(x) - \frac 2{p+1}\ird V(x)u_0^{p+1}\, dx\right) u_0=\|u_0\|_2^2 V(x)u_0^{p}.
\end{displaymath}
We set
\begin{displaymath}
	\lambda = \frac 2{p+1}\ird V(x)u_0^{p+1}\, dx, \quad \mu = \|u_0\|_2^2.
\end{displaymath}
With some simple computations, the function $\bar u=\mu^{\frac1{p-1}}u_0$ solves
\begin{displaymath}
	-\Delta \bar u + \left(V(x) - \lambda \right) \bar u= V(x)\bar u^{p}
\end{displaymath}
and the relation between $\l$ and $\bar u$ is obtained by direct computations.

Finally, since $H_V^1(\R^2)\hookrightarrow H^1(\R^2)$ and locally $V$ is strictly positive and bounded, the regularity of $\bar u$ follows as usual.

\end{proof}

\begin{remark}\label{re:th}
Some comments are in order about Theorem \ref{th:main} and its proof. Our result, with the presence of the parameter $\l$, is a direct consequence of the choice of  functional $I$ defined in \eqref{I} and, in particular, of the presence of the multiplicative factor $L^2$-norm. This choice is, for sure, quite tricky, but it is strictly related to lack of the compact embedding of $\HV$ into $L_V^{p+1}(\RD)$. 
Removing the multiplicative factor $L^2$-norm, we can still avoid the vanishing and the dichotomy for a minimizing sequence. Unfortunately, the sequence $\lbrace \xi_n\rbrace_n$ might well diverge to infinity: the minimizing sequence could behave as the sequence described in the proof of Proposition \ref{pr:general}.
\end{remark}

The difficulties described in Remark \ref{re:th} basically disappear when the compact embedding holds. \\
In view of this, we list the following three possibilities
	\begin{enumerate}[label=($H_\arabic{*}$),ref=$H_\arabic{*}$]
\setcounter{enumi}{0}
		\item\label{H1} $N=1, p>2$, $V$ is coercive, $W(x)\le CV(x)$   for some $C>0$ and any $x\in\R$ and \eqref{gradv1} holds for $V$;
		\item\label{H2} $N\ge 1$, $p$ satisfying\eqref{tau} with  $\alpha\in (0,1)$, $V$ and $W$ such that  $V$ is coercive and and \eqref{VWa} holds;
		\item\label{H3} $N\ge 2$, $V$ and $W$ are radial functions   such that \eqref{VWx} holds for $V$, $W$ and $p\in (2,2^*)$ ($2^*=\infty$ if $N=2$).
	\end{enumerate}
Taking into account Theorems \ref{thvwa}, \ref{th:main1} and \ref{thrad} and Remark \ref{re:VW1}, we have the following 
\begin{corollary} \label{th:main2}
Assume that $V\colon \RN\to\R$ and  $W\colon \RN\to\R$ are potentials in $\mathscr{C}^*(\RN)$ and one among \eqref{H1}, \eqref{H2} and \eqref{H3} holds.
Then there exists a nontrivial solution $u \in H^1_{V}(\RN)$ (actually $u \in H^1_{V,{\rm rad}}(\RN)$ if \eqref{H3} holds) of
\[
-\Delta u +V(x)  u=W(x)|u|^{p-1}u \qquad\text{ in }\RN.
\]
\end{corollary}

\bibliographystyle{amsplain}
\nocite{*}
\bibliography{aps}

\end{document}